\documentclass[english]{amsart}
\usepackage[T1]{fontenc}
\usepackage[latin9]{inputenc}
\usepackage{url}
\usepackage{amsthm}
\usepackage{amssymb}

\makeatletter
\numberwithin{equation}{section}
\numberwithin{figure}{section}

\usepackage{amsthm}\usepackage{amsfonts}

\subjclass{15A52}

\theoremstyle{plain}
  \newtheorem{theorem}{Theorem}

  \newtheorem{corollary}[theorem]{Corollary}

\theoremstyle{definition}
  \newtheorem{definition}[theorem]{Definition}

\include{psfig}

\usepackage{babel}

\makeatother

\usepackage{babel}
\begin{document}

\title[Special bases for the vector space of square matrices]{Special bases for the vector space of square matrices}

\author{Edinah K. Gnang}

\address{Department of Computer Science, Rutgers, New Brunswick NJ 08854-8019}

\email{gnang@cs.rutgers.edu}

\thanks{Edinah K. Gnang is supported by NSF grant DGE-0549115.}
\begin{abstract}
We describe families of complete orthogonal bases of full rank matrices
which span the vector spaces of square matrices. The proposed bases
generalise non-trivially the Pauli matrice while shedding light on
their algebraic properties. Finally we introduce the notion $k$-pseudo-closure
for orthogonal bases spanning vector subspaces of square matrices
and discuss their connections with hadamard matrices. 
\end{abstract}
\maketitle
\setcounter{tocdepth}{1} % \tableofcontents

\section{Introduction}

It is a common practice when manipulating $n\times n$ matrices to
express them as linear combinations of basis elements from the complete
orthonormal basis 
\[
B_{n}=\left\{ \mathbf{e}_{i}\mathbf{e}_{j}^{t}\right\} _{0\le i,j<n}
\]
 where the set $\left\{ \mathbf{e}_{i}\right\} _{0\le i<n}$ denotes
the canonical euclidean basis of $n$-dimensional column vectors.
Although a prevalent choice as canonical basis for expressing square
matrices, we observe that the basis $B_{n}$ is not as well behaved
with respect to the matrix product operation, more precisely 
\[
\left(\mathbf{e}_{i_{1}}\mathbf{e}_{j_{1}}^{t}\right)\left(\mathbf{e}_{i_{2}}\mathbf{e}_{j_{2}}^{t}\right)=\begin{cases}
\begin{array}{cc}
\mathbf{e}_{i_{1}}\mathbf{e}_{j_{2}}^{t} & \mbox{if }j_{1}=i_{2}\\
\mathbf{0} & \mbox{otherwise}
\end{array}\end{cases},
\]
 in addition elements of $B_{n}$ are rank $1$ matrices. An important
insight confered by linear algebra is the fact that the choice of
a basis for a vector spaces is crucial to the investigation of properties
of a particular set of vectors. Incidentally we discuss here alternative
complete orthogonal bases of full rank matrices which in some instances
are also pseudo-closed under matrix product as defined bellow.

\begin{definition}{[}Pseudo-closure{]}\label{def:Pseudo-Closure}
Let $k\geq1$ and $G$ denote an orthogonal basis for a subspace of
square matrices. $G$ is said to be pseudo-closed of order $k$ with
respect to the matrix product operation or simply a $k$-pseudo-closed
basis if

\begin{equation}
\forall\mathbf{A},\mathbf{B}\in G,\ \begin{cases}
\begin{array}{c}
\mathbf{A}^{\star^{k+1}}=\mathbf{A}\\
\mathbf{B}^{\star^{k+1}}=\mathbf{B}
\end{array},\end{cases}\exists\ \mathbf{D}\mbox{ a diagonal matrix s.t. }\mathbf{D}\,\left(\mathbf{A}\,\mathbf{B}\right)\in G\ \mbox{and}\ \mathbf{D}^{k}=\mathbf{I}.
\end{equation}

\end{definition} where for $\mathbf{A}=\left(a_{i,j}\right)_{0\le i,j<n}$
we define $\mathbf{A}^{\star^{k+1}}:=\left(a_{i,j}^{k+1}\right)_{0\le i,j<n}$.

\section{Orthonormal basis Induced by the Unitary group}

Let $\left\{ \mathbf{e}_{i}\right\} _{0\le i<n}$ denote the canonical
euclidean basis of column vectors. We define the \emph{fourier complete
orthogonal basis} to be the set $F_{n}$,

\begin{equation}
F_{n}\::=\left\{ \mathbf{B}\left(k,\, l\right)=\sum_{0\le j<n}\left(\mathbf{e}_{j}\mathbf{e}_{\left\{ j+k\:\mbox{mod}\: n\right\} }^{t}\right)\,\exp\left\{ i\frac{2\pi}{n}j\, l\right\} \right\} _{0\le k,l<n}.
\end{equation}
We think of the set of matrices as an inner-product space with the
inner-product being defined as follows

\[
\left\langle \mathbf{A},\:\mathbf{M}\right\rangle \::=\mbox{Tr}\left\{ \mathbf{A}\,\mathbf{M}^{\dagger}\right\} 
\]
 hence 
\begin{equation}
\forall\;\mathbf{A}\in\mathbb{C}^{n\times n},\quad\mathbf{A}=\sum_{0\le k,l<n}n^{-1}\left\langle \mathbf{A},\:\mathbf{B}\left(k,\, l\right)\right\rangle \mathbf{B}\left(k,\, l\right).
\end{equation}
 The proposed fourier basis generalizes the Pauli Matrices. In contrast
to the conventional canonical orthonormal matrix basis $B_{n}$, the
basis elements of the fourier basis are each full rank and it follows
from their definition that they consititute a $n$-pseudo-closed complete
orthogonal basis spanning the vector space of $n\times n$ matrices.
It also follows from the defintion of $F_{n}$ that it's elements
generate a finite mulitplicative group of matrices of order bounded
by $n^{n+1}$. Furthermore the group generated by $F_{n}$ is isomorphic
to a subgroup of $S_{n^{2}}$ which we represente here using matrices
in $\left\{ 0,1\right\} ^{n^{2}\times n^{2}}$. Let $T$ denote the
group Isomorphism 
\[
T\::\:\mathbb{C}^{n\times n}\rightarrow\mathbb{C}^{n^{2}\times n^{2}}
\]
 
\begin{equation}
T\left(\mathbf{B}\left(k,\, l\right)\right)=\sum_{0\le u<n}\left(\mathbf{e}_{u}\mathbf{e}_{\left\{ u+k\:\mbox{mod}\: n\right\} }^{t}\right)\otimes\left(\sum_{0\le v<n}\left(\mathbf{e}_{v}\mathbf{e}_{\left\{ v+j\times l\:\mbox{mod}\: n\right\} }^{t}\right)\right)=\mathbf{F}\left(k,\, l\right),
\end{equation}
 the isomorphism is naturally extended to the set of all $n\times n$
matrice as follows 
\begin{equation}
\forall\mathbf{A}\in\mathbb{C}^{n\times n},\quad T\left(\mathbf{A}\right)=\sum_{0\le k,l<n}n^{-1}\left\langle \mathbf{A},\:\mathbf{B}\left(k,\, l\right)\right\rangle \mathbf{F}\left(k,\, l\right).
\end{equation}
 so that 
\begin{equation}
\begin{cases}
\begin{array}{c}
T\left(\mathbf{A}_{1}+\mathbf{A}_{2}\right)=T\left(\mathbf{A}_{1}\right)\times T\left(\mathbf{A}_{2}\right)\\
T\left(\mathbf{A}_{1}\times\mathbf{A}_{2}\right)=T\left(\mathbf{A}_{1}\right)\times T\left(\mathbf{A}_{2}\right)
\end{array} & \forall\mathbf{A}_{1},\mathbf{A}_{2}\in\mathbb{C}^{n\times n}\end{cases}.
\end{equation}
 We also note that the matrix $\sum_{0\le k<n}\mathbf{B}\left(k,\, k\right)$
is unitary and correponds to the Discrete Fourier Transform (DFT)
matrix. from which it follows that the matrix product of the DFT matrix
with an arbitrary matrix $\mathbf{A}$ with entries in $\mathbb{Q}\left[e^{i\frac{2\pi}{3}}\right]$
can be recovered without using complex numbers at all.

Incidentaly there are families of complete orthogonal matrix basis
analogous to the fourier basis associated with arbitrary unitary matrices
and expressed by 
\begin{equation}
\left\{ \mathbf{Q}_{\mathbf{U}}\left(k,\, l\right)=\sum_{0\le j<n}\left(\mathbf{e}_{j}\mathbf{e}_{\left\{ j+k\:\mbox{mod}\: n\right\} }^{t}\right)u_{j\, l}\right\} _{0\le k,l<n}
\end{equation}
 where 
\[
\mathbf{U}\,\mathbf{U}^{\dagger}=\mathbf{I}
\]
 It also follows from this observation that for a hadamard matrix
$\mathbf{H}$, the orthonormal basis 
\begin{equation}
\left\{ \mathbf{Q}_{\mathbf{H}}\left(k,\, l\right)=\sum_{0\le j<n}\left(\mathbf{e}_{j}\mathbf{e}_{\left\{ j+k\:\mbox{mod}\: n\right\} }^{t}\right)h_{j\, l}\right\} _{0\le k,l<n}
\end{equation}
 is a $2$-pseudo-closed complete orthogonal basis spanning the set
of $n\times n$ matrices.

\begin{theorem}\label{2-pseudo-closure} Let $M_{n}$ denote the
vector space of $n\times n$ matrices with complex entries there exist
a set of $2$-pseudo-close complete orthgonal basis of full rank matrices
$\mathcal{B}$ which spans $M_{n}$ if and only if 
\begin{equation}
\exists\mathbf{H}\in M_{n}\mbox{ such that }\mathbf{H}^{T}\mathbf{H}=\mathbf{I}\mbox{ and }\mathbf{H}\star\mathbf{H}=\mathbf{1}_{n\times n}.
\end{equation}
 \end{theorem}

Similarly to the fourier basis case, the complete orthogonal matrix
basis associated with a hadamard matrix $\mathbf{H}$ generate a finite
matrix group $\mathcal{H}$ of order bouded by $n2^{n}$ which we
call \emph{the Hadamard group}. The Hadamard group is isomorphic to
a subgroup of $S_{2n}$ and the group isomorphism $R$ is described
bellow 

\begin{corollary}\label{Isomorphism} It follows that hadamard matrix
$\mathbf{H}$ induces a non trivial injective map 
\[
R:\: M_{n}\rightarrow M_{2n}
\]
 
\[
\forall\mathbf{A}\in\mathbb{C}^{n\times n},
\]
 
\begin{equation}
R\left(\mathbf{A}\right)=\sum_{0\le k,l<n}n^{-1}\left\langle \mathbf{A},\:\mathbf{Q}_{\mathbf{H}}\left(k,\, l\right)\right\rangle \sum_{0\le j<n}\left(\mathbf{e}_{j}\mathbf{e}_{\left\{ j+k\:\mbox{mod}\: n\right\} }^{t}\right)\otimes\begin{cases}
\begin{array}{cc}
\left(\begin{array}{cc}
1 & 0\\
0 & 1
\end{array}\right) & \mbox{if }h_{j\, l}=1\\
\left(\begin{array}{cc}
0 & 1\\
1 & 0
\end{array}\right) & \mbox{if }h_{j\, l}=-1
\end{array}\end{cases}.
\end{equation}
 so that 
\begin{equation}
\begin{cases}
\begin{array}{c}
R\left(\mathbf{A}+\mathbf{B}\right)=R\left(\mathbf{A}\right)+R\left(\mathbf{B}\right)\\
R\left(\mathbf{A}\times\mathbf{B}\right)=R\left(\mathbf{A}\right)\times R\left(\mathbf{B}\right)
\end{array}\end{cases}.
\end{equation}
 \end{corollary}

\begin{proof} The fact that Hadamard matrices can be used to construct
$2$-pseudo-close complete orthogonal basis is immediate from the
discussion in section{[}1{]}, consequently our proof shall focus on
showing that the existence of a $2$-pseudo-close complete orthogonal
basis of full rank matrices $\mathcal{B}=\left\{ \mathbf{N}\left(k,\, l\right)\right\} _{0\le k,l<n}$,
implies the existence of an $n\times n$ hadamard matrix.

If $\mathcal{B}$ is a complete basis then it must also express diagonal
matrices, hence for a diagonal matrix $\mathbf{D}$ we have 
\begin{equation}
\mathbf{D}=\sum_{0\le k,l<n}\frac{1}{\left\Vert \mathbf{N}\left(k,\, l\right)\right\Vert _{\ell_{2}}^{2}}\left\langle \mathbf{D},\:\mathbf{N}\left(k,\, l\right)\right\rangle \mathbf{N}\left(k,\, l\right).
\end{equation}
 We recall that the defining property for diagonal matrices is the
fact that 
\begin{equation}
\forall m>2\quad\mathbf{D}^{m}=\mathbf{D}^{\star^{m}}
\end{equation}
from which it follows that the elements of $\mathcal{B}$ for which
$\left\langle \mathbf{D},\:\mathbf{N}\left(k,\, l\right)\right\rangle \ne0$
must also be diagonal matrices. Futhermore since the basis element
have entries belong to the set $\left\{ 0,\pm1\right\} $ and most
importantly are full rank and it follows that the diagonal entries
should be non zero and the diagonal elements of $\mathcal{B}$ should
span a vector space of dimension $n$ , which completes the proof
$\square$. \end{proof}

\section{Conclusion}

We have discussed here a variety of matrix bases and illustrated how
the notion of $k$-pseudo closure for complete orthogonal matrix bases
ties together closure properties with respect to the matrix product
operation so often associated with groups on one hand and complete
orthogonal basis commonly associated with vector spaces on the other
hand. We point out that the $k$-pseudo-close complete matrix basis
generalize Pauli matrices and simultaneously provide us with an alternative
approach to generalizing the algebra of quaternions and octonions.
Furthermore by analogy to discrete fourier analysis we argue that
these basis suggest a natural framework for matrix fourier transform
providing us with a choice of basis from which on might select the
one which is best suited to some particular application. Finally from
an algorithmic point of view, recalling the fact that the fast fourier
transform plays a crucial role for fast integer mulitplication algorithm,
It might be of interest to investigate whether matrix fourier transform
and their corresponding convolution products also suggest efficient
algorithms for matrix multiplication.

\section{acknowledgement}

The author is indebted to Prof. Doron Zeilberger, Prof. Vladimir Retakh,
Prof. Ahmed Elgammal and Prof. Henry Cohn for insightful discussions
and precious advice.

\end{document}